\documentclass[10pt]{article}

\usepackage{amsmath,amssymb,amsthm,amsfonts}
\usepackage[letterpaper]{geometry}
\usepackage{multicol}

\theoremstyle{definition}
\newtheorem{definition}{Definition}[section]
\newtheorem{proposition}[definition]{Proposition}
\newtheorem{theorem}[definition]{Theorem}
\newtheorem{lemma}[definition]{Lemma}

\theoremstyle{remark}
\newtheorem{remark}[definition]{Remark}

\newcommand{\C}{\mathbb{C}}
\newcommand{\even}{\bar{0}}
\newcommand{\odd}{\bar{1}}
\newcommand{\Z}{\mathbb{Z}}
\newcommand{\TT}{\mathrm{T}}
\newcommand{\LL}{\mathrm{L}}
\newcommand{\UU}{\mathrm{U}}
\newcommand{\GG}{\mathrm{G}}
\newcommand{\QQ}{\mathrm{Q}}
\newcommand{\JJ}{\mathrm{J}}

\begin{document}
\title{On  twisted large $N=4$ conformal superalgebras}
\author{Zhihua Chang$^1$ and Arturo Pianzola$^{1,2}$}
\maketitle

$^1$ Department of Mathematical \& Statistical Science, University of Alberta, Edmonton, Alberta, T6G 2G1, Canada.

$^2$ Centro de Altos Estudios en Ciencias Exactas, Avenida de Mayo 866, (1084), Buenos Aires, Argentina.

\begin{abstract}
We explicitly compute the automorphism group of the large $N=4$ conformal superalgebra and classify the twisted loop conformal superalgebras based on the large $N=4$ conformal superalgebra. By considering the corresponding superconformal Lie algebras, we validate the existence of only two (up to isomorphism) such algebras as described in the Physics literature.
\bigskip

\bigskip

\end{abstract}

\allowdisplaybreaks[2]

\section{Introduction}
\label{sec:intro}

Superconformal Lie algebras are crucial objects in the study of conformal field theories. From a mathematical point of view, a superconformal Lie algebra\footnote{Based on the formal definition given in \cite[5.9]{K}, a superconformal Lie algebra satisfies certain simplicity assumptions. We do not, however, make any such simplicity assumption  in this paper.}
is an infinite dimensional Lie superalgebra corresponding to a  (twisted) loop conformal superalgebra.

Let us go into some detail to clarify this point. According to the axiomatic definition in \cite{K}, a {\it conformal superalgebra} over the field $\Bbb{C}$ of complex numbers consists of
\begin{itemize}
\item a $\Bbb{Z}/2\Bbb{Z}$--graded $\Bbb{C}$--vector space $\mathcal{A}$,
\item a $\Bbb{C}$--bilinear product ${-}_{(n)}{-}$ on $\mathcal{A}$ for each non-negative integer $n$, and
\item a $\Bbb{C}$--linear operator $\partial$ on $\mathcal{A}$ which acts as a derivation for each of the $n$- products.
\end{itemize}
Given a conformal superalgebra $\mathcal{A}$ over $\Bbb{C}$ and an automorphism $\sigma$ of $\mathcal{A}$ of order $m$, $\mathcal{A}\otimes_{\Bbb{C}}\Bbb{C}[t^{\pm1/m}]$ can be equipped with a conformal superalgebra structure given by
\begin{equation}
\widehat{\partial}(a\otimes f)=\partial(a)\otimes f+a\otimes\delta_t(f),\label{eq:loopder}
\end{equation}
and
\begin{equation}
(a\otimes f)_{(n)}(b\otimes g)=\sum\limits_{j\geqslant0}(a_{(n+j)}b)\otimes \delta_t^{(j)}(f)g,\label{eq:loopprod}
\end{equation}
for $a,b\in\mathcal{A}$, $f,g\in\Bbb{C}[t^{\pm1/m}]$ and $n\geqslant0$, where $\delta_t$ is the derivative with respect to $t$ and $\delta_t^{(j)}=\delta_t^j/j!$. Further, $\mathcal{A}\otimes_{\Bbb{C}}\Bbb{C}[t^{\pm1/m}]$ has a sub conformal superalgebra
$$\mathcal{L}(\mathcal{A},\sigma):=\bigoplus\limits_{i\in\Bbb{Z}}\mathcal{A}_i\otimes \Bbb{C} t^{i/m}\subseteq\mathcal{A}\otimes_{\Bbb{C}}\Bbb{C}[t^{\pm1/m}],$$
where $\mathcal{A}_i:=\{a\in\mathcal{A}|\sigma(a)=\zeta_m^ia\}$ and $\zeta_m = e^{\frac{i2\pi}{m}}$ is the standard $m$-th primitive root of unity.
The     conformal superalgebra $\mathcal{L}(\mathcal{A},\sigma)$ is called the {\it twisted loop conformal superalgebra} based on $\mathcal{A}$ with respect to $\sigma$. These are the conformal analogues of the twisted affine Kac-Moody Lie algebras (derived modulo their centre). $\mathcal{L}(\mathcal{A},\sigma)$ yields a Lie superalgebra (in general infinite dimensional, and usually referred to as a superconformal [Lie] algebra in Physics)
$$\mathrm{Alg}(\mathcal{A},\sigma)=\mathcal{L}(\mathcal{A},\sigma)/\widehat{\partial}\mathcal{L}(\mathcal{A},\sigma)$$
with the Lie super-bracket induced from the $0$-th product of $\mathcal{L}(\mathcal{A},\sigma)$. Up to isomorphisms, central extensions of these Lie superalgebras $\mathrm{Alg}(\mathcal{A},\sigma)$'s give the superconformal Lie algebras that are of interest in physics.

In order to classify twisted loop conformal superalgebras, the theory of differential conformal superalgebras was developed in \cite{KLP}. Based on this point of view, $\mathcal{L}(\mathcal{A},\sigma)$ has not only a complex conformal superalgebra structure but also a differential conformal superalgebra over the differential ring $(\Bbb{C}[t^{\pm1}], \frac{d}{dt})$. The resulting differential conformal superalgebra structure over $(\mathbb{C}[t^{\pm1}],\frac{d}{dt})$ allows us to understand $\mathcal{L}(\mathcal{A},\sigma)$ as a twisted form of $\mathcal{L}(\mathcal{A},\mathrm{id})$ with respect to the \'{e}tale extension $\Bbb{C}[t^{\pm1}]\rightarrow\Bbb{C}[t^{\pm1/m}]$. With some finiteness assumption on $\mathcal{A}$, these twisted forms can be classified in terms of non-abelian continuous cohomology set $\mathrm{H}_{\mathrm{ct}}^1\left(\widehat{\Bbb{Z}},
\mathrm{Aut}_{\widehat{\mathcal{D}}\text{-conf}}(\mathcal{A}\otimes_{\Bbb{C}}\widehat{\mathcal{D}})\right)$, where $\widehat{\Bbb{Z}}=\lim\limits_{\longleftarrow}\Bbb{Z}/m\Bbb{Z}$, $\widehat{\mathcal{D}}=(\lim\limits_{\longrightarrow}\Bbb{C}[t^{\pm1/m}],\frac{d}{dt})$ and $\mathrm{Aut}_{\widehat{\mathcal{D}}\text{-conf}}(\mathcal{A}\otimes_{\Bbb{C}}\widehat{\mathcal{D}})$ is the automorphism group of the $\widehat{\mathcal{D}}$--conformal superalgebra $\mathcal{A}\otimes_{\Bbb{C}}\widehat{\mathcal{D}}$.

The above theory has been applied to the $N=1,2,3$ and the small $N=4$ conformal superalgebras. Concretely, the twisted loop conformal superalgebras corresponding to the $N=2$ and small $N=4$ superconformal algebras have been classified in \cite{KLP}. The same classification for $N=3$ has been obtained in \cite{CP2011}. This work also provides  a detail investigation of the automorphism group functors of the $N=1,2,3$ conformal superalgebras. The structure of automorphism group functor of the small $N=4$ conformal superalgebras has been explicitly determined in \cite{C2012}.

Besides the $N=1,2,3$ and the small $N=4$ superconformal algebras, there is another family of superconformal algebras, called large (or big, or maximal) $N=4$ superconformal algebras, which are of interest in two dimensional conformal fields theories. They were discovered in \cite{STV} and have inspired subsequent work such as  \cite{DST}, \cite{R} and \cite{V}. The global and local automorphisms of the large $N=4$ superconformal algebras have been studied in \cite{DST}, and it is based on this that the twisted large $N=4$ superconformal algebras have been described in \cite{V}. In fact, the same method has been employed to deal with twisted $N=2,3,4$ superconformal algebras in \cite{SS1987}.

In this paper, we focus on the large $N=4$ conformal superalgebra, which is the conformal superalgebra $\mathcal{A}$ associated with the large $N=4$ superconformal Lie algebras. We will classify the twisted loop conformal superalgeras based on $\mathcal{A}$ by employing the methods developed in \cite{KLP}. To accomplish this, we first re-formulate the generators and the relations of $\mathcal{A}$ in Section~\ref{sec:conformalalgebra}. Then we explicitly describe the corresponding automorphism group in Section~\ref{sec:automorphism} (a result which we believe is of its own interest), and determine the relevant non-abelian cohomology needed for the classification of twisted forms in Section~\ref{sec:twistedform}. It is relevant to point out that our work shows that the automorphism group of the large $N=4$ superconformal algebra is in fact {\it larger} than the one described in the Physics literature (see \cite{DST}. See also Remark \ref{miss}). It may be that the new automophisms have no physical meaning (hence omitted) , or that they were simply missed.

We finish the paper by considering the (centreless) superconformal algebras corresponding to the two non-isomorphic twisted loop conformal superalgebras based on $\mathcal{A}$ in Section~\ref{sec:superalgebra}. We give an explicit description of the twisted Lie superalgebra $\mathrm{Alg}(\mathcal{A},\omega)$ which may be of interest to physicists. The given generators and relations of $\mathrm{Alg}(\mathcal{A},\omega)$ are in fact quite natural (but they can only be ``seen'' because the twisted construction was carried at the level of conformal superalgebras). As a consequence of the above, it follows that the two large $N=4$ superconformal algebras (ignoring central extensions) considered heretofore in the Physics literature are indeed (up to isomorphism) the only two such algebras. Furthermore, these two algebras are distinct (more precisely, not isomorphic to each other).\bigskip

\noindent{\bf Notation:} Throughout this paper, $\frak{sl}_2(\Bbb{C})$ and $\Bbb{C}^{2\times2}$ will denote the Lie algebra of $2\times2$ matrices with trace zero and the set of all $2\times 2$ matrices. As customary $\mathbf{SL}_2$  denote the group scheme of $2\times2$-matrices of determinant $1$, and $\mathbf{G}_a$ and $\mathbf{G}_{\mathrm{m}}$  the additive and multiplicative group scheme respectively (all group scheme over $\Bbb{C}$).

We will use $D$ and $D_m$ to denote the rings $\Bbb{C}[t^{\pm1}]$ and $\Bbb{C}[^{\pm1/m}]$, respectively; while we set $\widehat{D}:=\lim\limits_{\longrightarrow}D_m$. The differential operator $\delta_t:=\frac{d}{dt}$ acts on $D, D_m$ and $\widehat{D}$ in the usual way. The corresponding differential rings $(D,\delta_t), (D_m,\delta_t)$ and $(\widehat{D},\delta_t)$ are denoted by $\mathcal{D},\mathcal{D}_m$ and $\widehat{\mathcal{D}}$, respectively.

To simplify calculation in a conformal algebra $\mathcal{A}$, we also use the $\lambda$--bracket conviention
$$[a_\lambda b]=\sum\limits_{j\in\Bbb{Z}}\frac{1}{j!}\lambda^j a_{(j)}b,$$
for $a,b\in\mathcal{A}$.

\section{The large $N=4$ confornal superalgebras}
\label{sec:conformalalgebra}
The large $N=4$ superconformal algebras are a family of Lie superalgebras $\frak{g}(\gamma)$ parameterized by one parameter $\gamma\neq0,1$. More precisely, $\frak{g}(\gamma)=\frak{g}(\gamma)_{\even}\oplus\frak{g}(\gamma)_{\odd}$, where
\begin{align*}
\frak{g}(\gamma)_{\even}&=\mathrm{span}_{\mathbb{C}}\left\{\widetilde{c},\widetilde{\LL}_n,\widetilde{\TT}^{\pm i}_n, \widetilde{\UU}_n\middle| i=1,2,3, n\in\Z\right\},\\
\frak{g}(\gamma)_{\odd}&=\mathrm{span}_{\mathbb{C}}\left\{\widetilde{\GG}^p_s,\widetilde{\QQ}^p_s \middle|p=1,2,3,4,s\in1/2+\Z\right\}.
\end{align*}
$\widetilde{c}$ is a central element of $\frak{g}(\gamma)$ and the super-bracket on $\frak{g}(\gamma)$ is defined in \cite{STV} as follows:
\begin{center}
\begin{tabular}{lll}
\multicolumn{2}{l}{$[\widetilde{\LL}_m,\widetilde{\LL}_n]
=(m-n)\widetilde{\LL}_{m+n}+\delta_{m,-n}(m^3-m)\widetilde{c}/12,$}&
$[\widetilde{\LL}_m,\widetilde{\UU}_n]=-n\widetilde{\UU}_{m+n},$\\[1.5ex]
$[\widetilde{\LL}_m,\widetilde{\TT}^{\pm i}_n]=-n\widetilde{\TT}^{\pm i}_{m+n},$&
$[\widetilde{\LL}_m,\widetilde{\GG}^p_s]=\left(m/2-s\right)\widetilde{\GG}^p_{m+s},$&
$[\widetilde{\LL}_m,\widetilde{\QQ}^p_s]=-\left(\frac{1}{2}m+s\right)\widetilde{\QQ}^p_{m+s},$\\[1.5ex]
\multicolumn{2}{l}{$[\widetilde{\TT}^{+i}_m,\widetilde{\TT}^{+j}_n]=\epsilon_{ijk}\widetilde{\TT}^{+k}_{m+n}-m\delta_{ij}\delta_{m,-n} \widetilde{c}/(12\gamma),$}&
$[\widetilde{\TT}^{+i}_m,\widetilde{\TT}^{-j}_n]=[\widetilde{\TT}^{\pm i}_m,\widetilde{\UU}_n]=0,$\\[1.5ex]
\multicolumn{2}{l}{$[\widetilde{\TT}^{-i}_m,\widetilde{\TT}^{-j}_n]=\epsilon_{ijk}\widetilde{\TT}^{-k}_{m+n}-m\delta_{ij}\delta_{m,-n} \widetilde{c}/(12(1-\gamma)),$}&
$[\widetilde{\UU}_m,\widetilde{\UU}_n]=-m\delta_{m,-n}\widetilde{c}/(12\gamma(1-\gamma)),$\\[1.5ex]
\multicolumn{2}{l}{$[\widetilde{\TT}^{+i}_m,\widetilde{\GG}^p_s]=
\alpha^{+i}_{pq}(\widetilde{\GG}^q_{m+s}-2(1-\gamma)m\widetilde{\QQ}^q_{m+s}),$}&
$[\widetilde{\TT}^{-i}_m,\widetilde{\GG}^p_s]=\alpha^{-i}_{pq}(\widetilde{\GG}^q_{m+s}+2\gamma m\widetilde{\QQ}^q_{m+s}),$\\[1.5ex]
$[\widetilde{\TT}^{\pm i}_m,\widetilde{\QQ}^p_s]=\alpha^{\pm i}_{pq}\widetilde{\QQ}^q_{m+s},$&
$[\widetilde{\UU}_m, \widetilde{\GG}^p_s]=m\widetilde{\QQ}^p_{m+s},$&
$[\widetilde{\UU}_m, \widetilde{\QQ}^p_s]=0,$\\[1.5ex]
\multicolumn{2}{l}{$[\widetilde{\QQ}^p_r,\widetilde{\GG}^q_s]
=\delta_{pq}\widetilde{\UU}_{r+s}+2(\alpha^{+i}_{pq}\widetilde{\TT}^{+i}_{r+s}-\alpha^{-i}_{pq}\widetilde{\TT}^{-i}_{r+s}),$}&
$[\widetilde{\QQ}^p_r,\widetilde{\QQ}^q_s]=-\delta_{pq}\delta_{r,-s}\widetilde{c}/(12\gamma(1-\gamma)),$\\[1.5ex]
\multicolumn{3}{l}{$[\widetilde{\GG}^p_r,\widetilde{\GG}^q_s]=2\delta_{pq}\widetilde{\LL}_{r+s}
+4(s-r)(\gamma\alpha^{+i}_{pq}\widetilde{\TT}^{+i}_{r+s}+(1-\gamma)\alpha^{-i}_{pq}\widetilde{\TT}^{-i}_{r+s}) +\delta_{pq}\delta_{r,-s}(r^2-1/4) \widetilde{c}/3,$}\\
\end{tabular}
\end{center}
for $i,j=1,2,3$, $p,q=1,2,3,4$, $m,n\in\Z$, and $r,s\in1/2+\Z$, where $\alpha^{\pm i}$ are $4\times4$--matrices given by
$$\alpha^{\pm i}_{pq}=\pm\frac{1}{2}(\delta_{ip}\delta_{4q}-\delta_{iq}\delta_{4p})+\frac{1}{2}\epsilon_{iab}.$$
By setting
\begin{align*}
&\LL_n=\widetilde{\LL}_n+(\gamma-1/2)(n+1)\widetilde{\UU}_n,&&\TT^{\pm i}_n=\widetilde{\TT}^{\pm i}_n, &&\UU_n=\widetilde{\UU}_n,\\
&\GG^p_s=\widetilde{\GG}^p_s+2(\gamma-1/2)(s+1/2)\widetilde{\QQ}^p_s, &&\QQ^p_s=\widetilde{\QQ}^p_s,&&c=\widetilde{c}/(4\gamma(1-\gamma)),
\end{align*}
for $n\in\Bbb Z, s\in1/2+\Bbb Z$, the anti-commutative relations on $\frak{g}(\gamma)$ are written as
\begin{center}
\begin{tabular}{ll}
$[\LL_m,\LL_n]=(m-n)\LL_{m+n}+\delta_{m,-n}(m^3-m)c/12,$&
$[\LL_m,\TT^{\pm i}_n]=-n\TT^{\pm i}_{m+n},$\\[1.5ex]
$[\LL_m,\UU_n]=-n\UU_{m+n}-(\gamma-1/2)(m^2+m)\delta_{m,-n}c/3,$&
$[\TT^{+i}_m,\TT^{-j}_n]=0,$\\[1.5ex]
$[\TT^{+i}_m,\TT^{+j}_n]=\epsilon_{ijk}\TT^{+k}_{m+n}-(1-\gamma)m\delta_{ij}\delta_{m,-n}c/3,$&
$[\TT^{\pm i}_m,\UU_n]=0,$\\[1.5ex]
$[\TT^{-i}_m,\TT^{-j}_n]=\epsilon_{ijk}\TT^{-k}_{m+n}-\gamma m\delta_{ij}\delta_{m,-n}c/3,$&
$[\UU_m,\UU_n]=-m\delta_{m,-n}c/3,$\\[1.5ex]
$[\LL_m,\GG^p_s]=\left(m/2-s\right)\GG^p_{m+s},$&
$[\TT^{+i}_m,\GG^p_s]=\alpha^{+i}_{pq}(\GG^q_{m+s}-m\QQ^q_{m+s}),$\\[1.5ex]
$[\UU_m, \GG^p_s]=m\QQ^p_{m+s},$&
$[\TT^{-i}_m,\GG^p_s]=\alpha^{-i}_{pq}(\GG^q_{m+s}+ m\QQ^q_{m+s}),$\\[1.5ex]
$[\LL_m,\QQ^p_s]=-\left(\frac{1}{2}m+s\right)\QQ^p_{m+s},$&
$[\TT^{\pm i}_m,\QQ^p_s]=\alpha^{\pm i}_{pq}\QQ^q_{m+s},$\\[1.5ex]
$[\UU_m, \QQ^p_s]=0,$&
$[\QQ^p_r,\QQ^q_s]=-\delta_{pq}\delta_{r,-s}c/3,$\\[1.5ex]
\multicolumn{2}{l}{$[\QQ^p_r,\GG^q_s]
=\delta_{pq}\UU_{r+s}+
2(\alpha^{+i}_{pq}\TT^{+i}_{r+s}-\alpha^{-i}_{pq}\TT^{-i}_{r+s})
-2\delta_{pq}\delta_{r,-s}(\gamma-1/2)(s+1/2)c/3$,}\\[1.5ex]
\multicolumn{2}{l}{$[\GG^p_r,\GG^q_s]=2\delta_{pq}\LL_{r+s}
+2(s-r)(\alpha^{+i}_{pq}\TT^{+i}_{r+s}+\alpha^{-i}_{pq}\TT^{-i}_{r+s}) +\delta_{pq}\delta_{r,-s}(r^2-1/4)c/3,$}\\
\end{tabular}
\end{center}
for $i,j=1,2,3$, $p,q=1,2,3,4$, $m,n\in\Z$, and $r,s\in\frac{1}{2}+\Z$.

The Lie superalgebra $\frak{g}(\gamma)/\C c$ is called the {\it centreless core} of $\frak{g}(\gamma)$. They are referred also as centreless superconformal algebras. We observe that all the centreless cores $\frak{g}(\gamma)/\C c$ are isomorphic and denote this common Lie superalgebra by $\frak{g}$. Every $\frak{g}(\gamma)$ is a central extension of $\frak{g}$.

To the Lie superalgebra $\frak{g}$ one associates the conformal superalgebra $\mathcal{A}$, whose underlying $\Z/2\Z$--graded $\C[\partial]$--module is
$$\mathcal{A}=(\C[\partial]\otimes V_{\even})\oplus(\C[\partial]\otimes V_{\odd})$$
where $V_{\even}=\C\LL\oplus\bigoplus_{i=1}^3(\C\TT^i\oplus\C\TT^{-i})\oplus\C\UU$ and $V_{\odd}=\bigoplus_{p=1}^4(\C\GG^p\oplus\C\QQ^p)$.
The $\lambda$--bracket on $\mathcal{A}$ is given by:
\begin{center}
\begin{tabular}{lll}
$[\LL_\lambda\LL]=(\partial+2\lambda)\LL,$&
$[\LL_\lambda\UU]=(\partial+\lambda)\UU,$&
$[\LL_\lambda\TT^{\pm i}]=(\partial+\lambda)\TT^{\pm i},$\\[1.5ex]
$[{\TT^{\pm i}}_\lambda\TT^{\pm j}]=\epsilon_{ijk}\TT^{\pm k},$&
$[{\TT^{+i}}_\lambda\TT^{-j}]=0,$&
$[{\TT^{\pm i}}_\lambda\UU]=[\UU_\lambda\UU]=0,$\\[1.5ex]
$[\LL_\lambda\GG^p]=\left(\partial+\frac{3}{2}\lambda\right)\GG^p,$&
$[{\TT^{+i}}_\lambda\GG^p]=\alpha_{pq}^{+i}(\GG^q-\lambda\QQ^q),$&
$[{\TT^{\pm i}}_\lambda\QQ^p]=\alpha^{\pm i}_{pq}\QQ^q,$\\[1.5ex]
$[\LL_\lambda\QQ^p]=\left(\partial+\frac{1}{2}\lambda\right)\QQ^p,$&
$[{\TT^{-i}}_\lambda\GG^p]=\alpha_{pq}^{- i}(\GG^q+\lambda\QQ^q),$&
$[{\QQ^p}_\lambda\QQ^q]=0,$\\[1.5ex]
$[\UU_\lambda \GG^p]=\lambda\QQ^p,$&
\multicolumn{2}{l}{$[{\GG^p}_\lambda\GG^q]=2\delta_{pq}\LL
-2(\partial+2\lambda)(\alpha^{+i}_{pq}\TT^{+i}+\alpha^{-i}_{pq}\TT^{-i}),$}\\[1.5ex]
$[\UU_\lambda \QQ^p]=0,$&
\multicolumn{2}{l}{$[{\QQ^p}_\lambda\GG^q]=\delta_{pq}\UU+2(\alpha^{+i}_{pq}\TT^{+i}
-\alpha^{-i}_{pq}\TT^{-i}).$}
\end{tabular}
\end{center}

To simplify computations, we introduce the following notation:
\begin{align*}
\TT^+(X)&:=-\mathbf{i}(x_{12}+x_{21})\TT^{+1}+(x_{12}-x_{21})\TT^{+2}+2\mathbf{i} x_{11}\TT^{+3}\\
\TT^-(X)&:=-\mathbf{i}(x_{12}+x_{21})\TT^{-1}+(x_{12}-x_{21})\TT^{-2}+2\mathbf{i} x_{11}\TT^{-3}\\
\GG(M)&:=(u_{12}+u_{21})\GG^1+\mathbf{i}(u_{12}-u_{21})\GG^2
-(u_{11}-u_{22})\GG^3-\mathbf{i}(u_{11}+u_{22})\GG^4\\
\QQ(M)&:=(u_{12}+u_{21})\QQ^1+\mathbf{i}(u_{12}-u_{21})\QQ^2
-(u_{11}-u_{22})\QQ^3-\mathbf{i}(u_{11}+u_{22})\QQ^4
\end{align*}
for $X=(x_{ij})\in\frak{sl}_2(\Bbb C), M=(u_{ij})\in\Bbb C^{2\times2}$, where $\mathbf{i}=\sqrt{-1}$. With this new notation, the $\lambda$-brackets on $\mathcal{A}$ are written as:
\begin{center}
\begin{tabular}{ll}
$[\LL_\lambda\LL]=(\partial+2\lambda)\LL,$&$[\LL_\lambda\UU]=(\partial+\lambda)\UU,$\\[1.5ex]
$[\LL_\lambda\TT^{\pm}(X)]=(\partial+\lambda)\TT^{\pm}(X),$&$[{\TT^{\pm}(X)}_\lambda\UU]=[\UU_\lambda\UU]=0,$\\[1.5ex]
$[{\TT^{\pm}(X)}_\lambda\TT^{\pm}(Y)]=\TT^{\pm}([X,Y]),$&$[{\TT^{+}(X)}_\lambda\TT^{-}(Y)]=0,$\\[1.5ex]
$[\LL_\lambda\GG(M)]=\left(\partial+\frac{3}{2}\lambda\right)\GG(M),$&$[\UU_\lambda \GG(M)]=\lambda\QQ(M),$\\[1.5ex]
$[\LL_\lambda\QQ(M)]=\left(\partial+\frac{1}{2}\lambda\right)\QQ(M),$&$[\UU_\lambda \QQ(M)]=[{\QQ(M)}_\lambda\QQ(N)]=0,$\\[1.5ex]
$[\TT^+(X)_\lambda\GG(M)]=\GG(XM)-\lambda\QQ(XM),$&$[{\TT^{+}(X)}_\lambda\QQ(M)]=\QQ(XM),$\\[1.5ex]
$[\TT^-(X)_\lambda\GG(M)]=-\GG(MX)-\lambda\QQ(MX),$&$[\TT^-(X)_\lambda\QQ(M)]=-\QQ(MX),$\\[1.5ex]
\multicolumn{2}{l}{$[{\GG(M)}_\lambda\GG(N)]=4\mathrm{tr}(MN^\dag)\LL
+(\partial+2\lambda)\left(\TT^+(MN^\dag-NM^\dag)+\TT^-(M^\dag N-N^\dag M)\right),$}\\[1.5ex]
\multicolumn{2}{l}{$[{\QQ(M)}_\lambda\GG(N)]=2\mathrm{tr}(MN^\dag)\UU-\TT^+(MN^\dag-NM^\dag)+\TT^-(M^\dag N-N^\dag M),$}
\end{tabular}
\end{center}
for $X,Y\in\frak{sl}_2(\Bbb C), M,N\in \Bbb C^{2\times2}$, where $M^\dag=\begin{pmatrix}-u_{22}&u_{12}\\u_{21}&-u_{11}\end{pmatrix}$ if $M=\begin{pmatrix}u_{11}&u_{12}\\u_{21}&u_{22}\end{pmatrix}$. In fact, the map $M\mapsto-M^\dag$ is a symplectic involution on the associative algebra of $2\times2$--matrices (cf. \cite[2.5]{Rowen}).

\section{Automorphisms of $\mathcal{A}$ and of $\mathcal{A}\otimes_{\Bbb C} {\widehat{\mathcal{D}}}$}
\label{sec:automorphism}
The main purpose of this paper is to classify twisted loop conformal superalgebras based on the large $N=4$ conformal superalgebra $\mathcal{A}$. A twisted loop conformal superalgebra $\mathcal{L}(\mathcal{A},\sigma)$ based on $\mathcal{A}$ can be viewed as a {\it $\widehat{\mathcal{D}}/\mathcal{D}$--form} of $\mathcal{A}\otimes_{\Bbb{C}}\mathcal{D} = \mathcal{L}(\mathcal{A},\mathrm{id})$, i.e.,
$$\mathcal{L}(\mathcal{A},\sigma)\otimes_{\mathcal{D}}\widehat{\mathcal{D}}
\cong\mathcal{A}\otimes_{\Bbb{C}}\widehat{\mathcal{D}}$$
where the above is an isomorphism of {\it differential} conformal superalgebras over $\widehat{\mathcal{D}}$, where the differential conformal superalgebras structure on these objects are given in the same way as those on $\mathcal{A}\otimes_{\Bbb{C}}\widehat{\mathcal{D}}$ given by (\ref{eq:loopder}) and (\ref{eq:loopprod}) above. Based on the results of \cite{KLP} (see also Proposition \ref{prop:Diso} below for details), $\widehat{\mathcal{D}}/\mathcal{D}$--forms of $\mathcal{A}\otimes_{\Bbb{C}}\widehat{\mathcal{D}}$ are classified in terms of $\mathrm{H}^1_{\mathrm{ct}}\left(\widehat{\Bbb{Z}},
\mathrm{Aut}_{\widehat{\mathcal{D}}\text{-conf}}(\mathcal{A}\otimes_{\Bbb{C}}\widehat{\mathcal{D}})\right)$. The computation of this $\mathrm{H}^1$ requires us to first determine the automorphism group $\mathrm{Aut}_{\widehat{\mathcal{D}}\text{-conf}}(\mathcal{A}\otimes_{\Bbb{C}}\widehat{\mathcal{D}})$.

We will determine the automorphism group $\mathrm{Aut}_{\widehat{\mathcal{D}}\text{-conf}}(\mathcal{A}\otimes_{\Bbb{C}}\widehat{\mathcal{D}})$ in this section by explicitly constructing all automorphism of $\mathcal{A}\otimes_{\Bbb{C}}\widehat{\mathcal{D}}$. To simplify notations, we write $\mathcal{A}_{\widehat{\mathcal{D}}}:=\mathcal{A}\otimes_{\Bbb{C}}\widehat{\mathcal{D}}$ for short. We always use $V$ to denote the $\Bbb{C}$--vector space spanned by $\{\LL,\TT^{\pm i},\UU,\GG^p,\QQ^p|i=1,2,3,p=1,2,3,4\}$. Note that $\mathcal{A}_{\widehat{\mathcal{D}}}=\Bbb{C}[\partial]\otimes_{\Bbb{C}}V\otimes_{\Bbb{C}}\widehat{D}$ as $\Bbb{C}$--vector spaces. $V$ can be identified with the subspace $1\otimes V\otimes 1$ in $\mathcal{A}_{\widehat{\mathcal{D}}}$. Hence, we also identify an element $\xi\in V$ with its image $1\otimes\xi\otimes1$ in $\mathcal{A}_{\widehat{\mathcal{D}}}$.

With the simplified notation of  Section~\ref{sec:conformalalgebra} and the above, we now proceed to construct automorphisms of the $\widehat{\mathcal{D}}$--conformal superalgebra $\mathcal{A}_{\widehat{\mathcal{D}}}$.

\begin{lemma}
\label{autoSL}
There is a group homomorphism
$$\iota_1:\mathbf{SL}_2(\widehat{D})\times\mathbf{SL}_2(\widehat{D})\rightarrow
\mathrm{Aut}_{\widehat{\mathcal{D}}\text{-conf}}(\mathcal{A}_{\widehat{\mathcal{D}}}),\quad
(A,B)\mapsto\theta_{A,B},$$
where $\theta_{A,B}$ is the automorphism of $\mathcal{A}_{\widehat{\mathcal{D}}}$ given by
\begin{align*}
&\theta_{A,B}(\LL)=\LL+\TT^+(\delta_t(A)A^{-1})+\TT^-(\delta_t(B)B^{-1}),&&\theta_{A,B}(\UU)=\UU,\\
&\theta_{A,B}(\TT^+(X))=\TT^+(AXA^{-1}),&&\theta_{A,B}(\TT^-(X))=\TT^-(BXB^{-1}),\\
&\theta_{A,B}(\GG(M))=\GG(AMB^{-1})-\QQ(\delta_t(A)MB^{-1}-AM\delta_t(B^{-1})),&&\theta_{A,B}(\QQ(M))=\QQ(AMB^{-1}),
\end{align*}
for $X\in\frak{sl}_2(\mathbb{C}), M\in\mathbb{C}^{2\times2}$.
\end{lemma}
\begin{proof}
Recall that the underlying $\widehat{D}$--module of $\mathcal{A}_{\widehat{\mathcal{D}}}$ is $\mathbb{C}[\partial]\otimes_{\mathbb{C}}V\otimes_{\mathbb{C}}\widehat{D}$. The formulas define a $\widehat{D}$--module homomorphism $V\otimes_{\mathbb{C}}\widehat{D}\rightarrow V\otimes_{\mathbb{C}}\widehat{D}$, which is uniquely extended to a $\widehat{D}$--module homomorphism $\theta_{A,B}:\mathcal{A}_{\widehat{\mathcal{D}}}\rightarrow\mathcal{A}_{\widehat{\mathcal{D}}}$ satisfying $\widehat{\partial}\circ\theta_{A,B}=\theta_{A,B}\circ\widehat{\partial}$.

Based on \cite[Lemma~3.1(ii)]{KLP}, it can be proved that $\theta_{A,B}$ is a homomorphism of $\widehat{\mathcal{D}}$--conformal superalgebras by checking
$$\theta_{A,B}([(\xi\otimes1)_\lambda(\eta\otimes1)])=[\theta_{A,B}(\xi\otimes1)_\lambda\theta_{A,B}(\eta\otimes1)]$$
for all $\xi,\eta\in V$. This can be accomplished through a direct computation. As an example, we show the proof for $\xi=\QQ(M)$ and $\eta=\GG(N)$ with $M,N\in\mathbb{C}^{2\times2}$.
\begin{align*}
&\theta_{A,B}([\QQ(M)_\lambda\GG(N)])\\
&=2\mathrm{tr}(MN^\dag)\theta_{A,B}(\UU)
-\theta_{A,B}(\TT^+(MN^\dag-NM^\dag))+\theta_{A,B}(\TT^-(M^\dag N-N^\dag M))\\
&=2\mathrm{tr}(MN^\dag)\UU-\TT^+(A(MN^\dag-NM^\dag)A^{-1})+\TT^-(B(M^\dag N-N^\dag M)B^{-1}),\\
&[\theta_{A,B}(\QQ(M))_\lambda\theta_{A,B}(\GG(N))]\\
&=[\QQ(AMB^{-1})_\lambda(\GG(ANB^{-1})-\QQ(\delta_t(A)NB^{-1}-AN\delta_t(B^{-1})))]\\
&=2\mathrm{tr}(AMB^{-1}(ANB^{-1})^\dag)\UU\\
&\quad-\TT^+(AMB^{-1}(ANB^{-1})^\dag-ANB^{-1}(AMB^{-1})^\dag)\\
&\quad+\TT^-((AMB^{-1})^\dag ANB^{-1}-(ANB^{-1})^\dag AMB^{-1})\\
&=2\mathrm{tr}(MN^\dag)\UU-\TT^+(A(MN^\dag-NM^\dag)A^{-1})+\TT^-(B(M^\dag N-N^\dag M)B^{-1}),
\end{align*}
where we use the facts that $(M_1M_2M_3)^\dag=M_3^\dag M_2^\dag M_1^\dag$ for $M_1,M_2,M_3\in\mathbb{C}^{2\times2}$ and that $A^{-1}=-A^\dag$ for $A\in\mathbf{SL}_2(\widehat{D})$.

A similar computation also shows that
$$\theta_{A_1,B_1}\circ\theta_{A_2,B_2}(\xi\otimes1)=\theta_{A_1A_2,B_1B_2}(\xi\otimes1),$$
for $A_1,A_2,B_1,B_2\in\mathbf{SL}_2(\widehat{D})$ and all $\xi\in V$. We thus deduce from \cite[Lemma~3.1(i)]{KLP} that $$\theta_{A_1,B_1}\circ\theta_{A_2,B_2}=\theta_{A_1A_2,B_1B_2}.$$
We also observe that $\theta_{I_2,I_2}=\mathrm{id}$, where $I_2$ is the identity matrix. Hence, the above equality implies that $\theta_{A,B}$ is invertible and $\iota_1:\mathbf{SL}_2(\widehat{D})\times\mathbf{SL}_2(\widehat{D})\rightarrow \mathrm{Aut}_{\widehat{\mathcal{D}}\text{-conf}}(\mathcal{A}_{\widehat{\mathcal{D}}}),
(A,B)\mapsto\theta_{A,B}$ is a group homomorphism.
\end{proof}

\begin{lemma}
\label{autoADD}
There is a group homomorphism
$$\iota_2:\mathbf{G}_a(\widehat{D})\rightarrow \mathrm{Aut}_{\widehat{D}\text{-conf}}(\mathcal{A}_{\widehat{\mathcal{D}}}),\quad
f\mapsto\tau_f,$$
where $\tau_f$ is the automorphism of $\mathcal{A}_{\widehat{\mathcal{D}}}$ defined by
\begin{align*}
&\tau_f(\LL)=\LL+\UU\otimes f,&&\tau_f(\TT^+(X))=\TT^+(X),&&\tau_f(\TT^-(X))=\TT^-(X),\\
&\tau_f(\UU)=\UU,&&\tau_f(\GG(M))=\GG(M)+\QQ(f M),&&\tau_f(\QQ(M))=\QQ(M),
\end{align*}
for $X\in\frak{sl}_2(\mathbb{C})$ and $M\in\mathbb{C}^{2\times2}$.
\end{lemma}
\begin{proof}
An analogous argument as in Lemma~\ref{autoSL} shows that the formulas define a homomorphism of $\widehat{\mathcal{D}}$--conformal superalgebras $\tau_f:\mathcal{A}_{\widehat{\mathcal{D}}}\rightarrow\mathcal{A}_{\widehat{\mathcal{D}}}$ for every $f\in\widehat{D}$ and $\tau_{f_1}\circ\tau_{f_2}=\tau_{f_1+f_2}$ for $f_1,f_2\in\widehat{D}$. Observing that $\tau_0=\mathrm{id}$, we obtain that $\tau_f$ has inverse $\tau_{-f}$ and $\iota_2$ is a group homomorphism.
\end{proof}

\begin{lemma}
\label{autoINV}
There is an automorphism $\omega$ of the $\widehat{\mathcal{D}}$--conformal superalgebra $\mathcal{A}_{\widehat{\mathcal{D}}}$ such that
\begin{align*}
&\omega(\LL)=\LL,&&\omega(\TT^+(X))=\TT^-(X),&&\omega(\TT^-(X))=\TT^+(X),\\
&\omega(\UU)=-\UU,&&\omega(\GG(M))=\GG(M^\dag),&&\omega(\QQ(M))=-\QQ(M^\dag),
\end{align*}
for $X\in\frak{sl}_2(\mathbb{C})$ and $M\in\mathbb{C}^{2\times2}$. In addition, $\omega^2=\mathrm{id}$.
\end{lemma}
\begin{proof}
The proof is similar to that of Lemma~\ref{autoSL}.
\end{proof}

\begin{lemma}
\label{autoCOM}\quad
\begin{enumerate}
\item For $A,B\in\mathbf{SL}_2(\widehat{D})$ and $f\in\mathbf{G}_a(\widehat{D})$, $\tau_f\circ\theta_{A,B}=\theta_{A,B}\circ\tau_f$.
\item For $A,B\in\mathbf{SL}_2(\widehat{D})$, $\omega\circ\theta_{A,B}\circ\omega=\theta_{B,A}$.
\item For $f\in\mathbf{G}_a(\widehat{D})$, $\omega\circ\tau_f\circ\omega=\tau_{-f}$.
\end{enumerate}
\end{lemma}
\begin{proof}
(i) From \cite[Lemma~3.1 (i)]{KLP}, it suffices to show
$$\tau_f\circ\theta_{A,B}(\xi\otimes1)=\theta_{A,B}\circ\tau_f(\xi\otimes1),$$
for all $\xi\in V$. This can be verified by a direct computation. Similarly for (ii) and (iii).
\end{proof}

\begin{theorem}
\label{autoTHM}
There is a group isomorphism:
$$\mathrm{Aut}_{\widehat{\mathcal{D}}\text{-conf}}(\mathcal{A}_{\widehat{\mathcal{D}}})\cong
\left(\frac{\mathbf{SL}_2(\widehat{D})\times\mathbf{SL}_2(\widehat{D})}{\langle(-I_2,-I_2)\rangle}
\times\mathbf{G}_a(\widehat{D})\right)\rtimes\Bbb Z/2\Bbb Z.$$
\end{theorem}
\begin{proof}
From Lemma~\ref{autoSL}-\ref{autoCOM}, there is a group homomorphism:
$$\iota:\left(\mathbf{SL}_2(\widehat{D})\times\mathbf{SL}_2(\widehat{D})\times\mathbf{G}_a(\widehat{D})\right)
\rtimes\Bbb Z/2\Bbb Z\rightarrow
\mathrm{Aut}_{\widehat{\mathcal{D}}\text{-conf}}(\mathcal{A}_{\widehat{\mathcal{D}}}),
\quad (A,B,f,\varepsilon)\mapsto\theta_{A,B}\circ\tau_f\circ\omega^{\varepsilon}.$$

We claim that $\iota$ is surjective, which is equivalent to say that every $\varphi\in\mathrm{Aut}_{\widehat{\mathcal{D}}\text{-conf}}(\mathcal{A}_{\widehat{\mathcal{D}}})$ is of the form $\theta_{A,B}\circ\tau_f\circ\omega^{\varepsilon}$ for some $A,B\in\mathbf{SL}_2(\widehat{D})$, $f\in\widehat{D}$, and $\varepsilon\in\{0,1\}$.

To prove the claim, we first consider the action of $\varphi$ on the even part $(\mathcal{A}_{\widehat{\mathcal{D}}})_{\even}$ of $\mathcal{A}_{\widehat{\mathcal{D}}}$. We observe
$$\mathcal{A}_{\even}=\mathbb{C}[\partial]\LL\oplus\mathcal{C}\oplus\mathbb{C}[\partial]\UU,$$
where $\mathcal{C}=\oplus_{i=1}^3(\Bbb C[\partial]\TT^{+i}\oplus\Bbb C[\partial]\TT^{-i})$.
Moreover, $\mathcal{B}:=\mathcal{C}\oplus\Bbb{C}[\partial]\UU$ is an ideal of $\mathcal{A}_{\even}$ and $\mathcal{B}$ is isomorphic to the current Lie conformal algebra $\mathrm{Cur}(\frak{sl}_2(\Bbb C)\oplus\frak{sl}_2(\Bbb C)\oplus \Bbb C)$ (see \cite[Example~2.7]{K} for its definition).

Next we show that $\varphi(\mathcal{B}_{\widehat{\mathcal{D}}})\subseteq\mathcal{B}_{\widehat{\mathcal{D}}}$. Let $\xi=\TT^{\pm i}, i=1,2,3$ or $\UU$. Write
$$\varphi(\xi)=\sum\limits_{m=0}^K\widehat{\partial}^m(\LL\otimes f_m)+\xi',$$
where $r_i\in\widehat{D}$ and $\xi'\in\mathcal{B}_{\widehat{\mathcal{D}}}$. Then
\begin{align*}
0&=\varphi([\xi_\lambda\xi])=[\varphi(\xi)_\lambda\varphi(\xi)]\\
&=\sum\limits_{m,n=0}^K(-\lambda)^m(\widehat{\partial}+\lambda)^n((\partial+2\lambda)\LL\otimes f_mf_n+2\LL\otimes\delta_t(f_m)f_n)\\
&\quad+\sum\limits_{m=0}^K(-\lambda)^m[(\LL\otimes f_m)_\lambda\xi']+\sum\limits_{n=0}^K(\widehat{\partial}+\lambda)^n[{\xi'}_\lambda(\LL\otimes f_n)]+[{\xi'}_\lambda\xi']
\end{align*}
Since $[(\LL\otimes f_m)_\lambda\xi'], [{\xi'}_\lambda(\LL\otimes f_n)], [{\xi'}_\lambda\xi']\in \Bbb C [\lambda]\otimes_{\Bbb C}\mathcal{B}_{\widehat{\mathcal{D}}}$, we deduce that
$$0=\sum\limits_{m,n=0}^K(-\lambda)^m(\widehat{\partial}+\lambda)^n((\partial+2\lambda)\LL\otimes f_mf_n+2\LL\otimes\delta_t(f_m)f_n).$$
Comparing the coefficients of $\lambda$ and noting that $\widehat{D}$ is an integral domain, we obtain $K=0$ and $f_0=0$, i.e., $\varphi(\xi)=\xi'\in\mathcal{B}_{\widehat{\mathcal{D}}}$. Since $\TT^{\pm i},i=1,2,3$ and $\UU$ generate $\mathcal{B}$ as a $\Bbb{C}[\partial]$--module and $\varphi$ is a $\widehat{D}$--module homomorphism satisfying $\varphi\circ\widehat{\partial}=\widehat{\partial}\circ\varphi$, we conclude that $\varphi(\mathcal{B}_{\widehat{\mathcal{D}}})\subseteq\mathcal{B}_{\widehat{\mathcal{D}}}$.

Furthermore, we deduce from ${\TT^{\pm i}}_{(0)}\TT^{\pm j}=\epsilon_{ijk}\TT^{\pm k}$ that
$$\epsilon_{ijk}\varphi(\TT^{\pm k})=\varphi(\TT^{\pm i})_{(0)}\varphi(\TT^{\pm j})\in (\mathcal{B}_{\widehat{\mathcal{D}}})_{(0)}(\mathcal{B}_{\widehat{\mathcal{D}}})\subseteq\mathcal{C}_{\widehat{\mathcal{D}}},$$
for $k=1,2,3$. It yields that $\varphi(\mathcal{C}_{\widehat{\mathcal{D}}})\subseteq\mathcal{C}_{\widehat{\mathcal{D}}}$.

Therefore, the restriction $\varphi|_{\mathcal{C}_{\widehat{\mathcal{D}}}}$ is an automorphism of $\mathcal{C}_{\widehat{\mathcal{D}}}$. It is known that $$\mathcal{C}_{\widehat{\mathcal{D}}}\cong\mathrm{Cur}(\frak{sl}_2(\Bbb{C})\oplus\frak{sl}_2(\Bbb{C}))_{\widehat{\mathcal{D}}}.$$ Given that $\frak{sl}_2(\Bbb{C})\oplus\frak{sl}_2(\Bbb{C})$ is a semisimple complex Lie algebra by \cite[Theorem~3.4]{KLP}, there are two elements $A,B\in\mathbf{GL}_2(\widehat{D})$ such that one of the two following conditions is satisfied
\begin{align}
\varphi(\TT^+(X))=\TT^+(AXA^{-1}),\text{ and }\varphi(\TT^-(X))=\TT^-(BXB^{-1}),\label{eq:keepsl2}\\
\varphi(\TT^+(X))=\TT^-(AXA^{-1}),\text{ and }\varphi(\TT^-(X))=\TT^+(BXB^{-1}).\label{eq:exsl2}
\end{align}
Since any unit of $\widehat{D}$ is a square there is no loss of generalization in assuming that $A,B\in\mathbf{SL}_2(\widehat{D})$. We take $\psi:=\varphi\circ\theta_{A,B}^{-1}$ if $\varphi$ satisfies (\ref{eq:keepsl2}), or $\psi:=\varphi\circ\omega\circ\theta_{A,B}^{-1}$ if $\varphi$ satisfies (\ref{eq:exsl2}). Then $\psi$ is also an automorphism of the $\widehat{\mathcal{D}}$--conformal superalgebra $\mathcal{A}_{\widehat{\mathcal{D}}}$ and always satisfies
\begin{equation}
\psi(\TT^+(X))=\TT^+(X),\text{ and }\psi(\TT^-(X))=\TT^-(X).\label{eq:fixsl2}
\end{equation}

To determine $\psi(\UU)$, we observe that $\Bbb{C}[\partial]\UU\otimes_{\Bbb{C}}\widehat{\mathcal{D}}$ is the center of $\mathcal{B}_{\widehat{\mathcal{D}}}$, which is preserved under $\psi$. Hence, $\psi(\UU)=P(\partial)\UU$, where $P(\partial)$ is a polynomial in the indeterminate $\partial$ with coefficient in $\widehat{D}$. Then the bijectivity of $\psi$ yields that $P(\partial)=f$ is an unit element in $\widehat{D}$, i.e., $\psi(\UU)=\UU\otimes f$ for an unit element $f\in\widehat{D}$.

Next we consider the action of $\psi$ on the odd part $(\mathcal{A}_{\widehat{\mathcal{D}}})_{\odd}$. Suppose
$$\psi(\GG(M))=\sum\limits_{m=0}^{K_1}\widehat{\partial}^m\GG(\nu_m(M))
+\sum\limits_{n=0}^{K_2}\widehat{\partial}^n\QQ(\nu'_n(M)),$$
where $\nu_m, \nu'_n:\Bbb{C}^{2\times2}\rightarrow\widehat{D}^{2\times2}$ are $\Bbb{C}$--linear maps. Then $\psi([\UU_\lambda\GG(M)])=[\psi(\UU)_\lambda\psi(\GG(M))]$ yields
$$\lambda\psi(\QQ(M))=\sum\limits_{m=0}^{K_1}(\widehat{\partial}+\lambda)^m(\lambda\QQ(r\nu_m(M))+\QQ(\delta_t(f)\nu_m(M))).$$
Comparing the coefficients of $\lambda$, we obtain $K_1=0$, i.e.,
\begin{align*}
\psi(\GG(M))&=\GG(\nu_0(M))+\sum\limits_{n=0}^{K_2}\widehat{\partial}^n\QQ(\nu'_n(M)),\\
\psi(\QQ(M))&=\QQ(f\nu_0(M)).
\end{align*}
Similarly, we deduce from $\psi([\TT^+(X)_\lambda\GG(M)])=[\psi(\TT^+(X))_\lambda\psi(\GG(M))]$ that $K_2=0$ and
\begin{equation}
\nu_0(XM)=X\nu_0(M), \quad\nu'_0(XM)=X\nu'_0(M),\quad f\nu_0(XM)=X\nu_0(M),\label{matrix1}
\end{equation}
for $X\in\frak{sl}_2(\Bbb{C}), M\in\Bbb{C}^{2\times2}$.  Further, $\psi([\TT^-(X)_\lambda\GG(M)])=[\psi(\TT^-(X))_\lambda\psi(\GG(M))]$ yields that
\begin{equation}
\nu_0(MX)=\nu_0(M)X, \quad\nu'_0(MX)=\nu'_0(M)X, \quad f\nu_0(MX)=\nu_0(M)X,\label{matrix2}
\end{equation}
for $X\in\frak{sl}_2(\Bbb{C}), M\in\Bbb{C}^{2\times2}$.

From (\ref{matrix1}) and (\ref{matrix2}), we conclude that $f=1$, and there are $g,g'\in\widehat{D}$ such that $\nu_0(M)=gM$ and $\nu_0'(M)=g'M$, i.e.,
\begin{align*}
\psi(\GG(M))=\GG(gM)+\QQ(g'M),\quad\psi(\QQ(M))=\QQ(gM),\quad\psi(\UU)=\UU.
\end{align*}

Finally, the equality
$$\psi(\QQ(M))_{(0)}\psi(\GG(N))=\psi(\QQ(M)_{(0)}\GG(N))$$
implies that $g=\pm1$; while the equality
$$\psi(\GG(M))_{(0)}\psi(\GG(N))=\psi(\GG(M)_{(0)}\GG(N))$$
yields $\psi(\LL)=\LL+\UU\otimes gg'$. Hence,
\begin{align}
\psi(\LL)=\LL+\UU\otimes g_0,\quad
\psi(\GG(M))=g(\GG(M)+\QQ(g_0M)),\quad
\psi(\QQ(M))=g\QQ(M),\quad
\psi(\UU)=\UU,\label{eq:odd}
\end{align}
where $g_0=gg'$. Summarizing (\ref{eq:fixsl2}) and (\ref{eq:odd}), we conclude that $\psi=\tau_{g_0}\circ\theta_{I,gI}$. Hence,
$$\varphi=\tau_{g_0}\circ\theta_{A,g B}=\theta_{A,g B}\circ\tau_{g_0},\text{ or }\varphi=\tau_{g_0}\circ\theta_{A,g B}\circ\omega=\theta_{A,g B}\circ\tau_{g_0}\circ\omega.$$
We complete the proof of the surjectivity of $\iota$.

Next we will determine the kernel of $\iota$. On one hand, it is obvious that $(-I_2,-I_2,0,0)\in\ker\iota_1$. On the other hand, we will show $(A,B,f,\varepsilon)\in\ker\iota$ will lead to $A=B=\pm I_2, f=0$ and $\varepsilon=0$. In fact, $(A,B,f,\varepsilon)\in\ker\iota$ is equivalent to $\theta_{A,B}\circ\tau_f\circ\omega^{\varepsilon}=\mathrm{id}$. Hence,
$$\UU=\theta_{A,B}\circ\tau_f\circ\omega^{\varepsilon}(\UU)=(-1)^{\varepsilon}\UU,$$
where $\varepsilon=0$ or $1$. It follows $\varepsilon=0$.

Similarly, $\theta_{A,B}\circ\tau_f(\LL)=\LL$ yields that $f=0$, and so we have
\begin{align*}
\theta_{A,B}(\TT^+(X))&=\TT^+(AXA^{-1})=\TT^+(X),\\
\theta_{A,B}(\TT^-(X))&=\TT^-(BXB^{-1})=\TT^-(X), \\
\theta_{A,B}(\QQ(M))&=\QQ(AMB^{-1})=\QQ(M),
\end{align*}
for all $X\in\frak{sl}_2(\mathbb{C})$ and all $M\in\mathbb{C}^{2\times2}$, i.e.,
$$AX=XA, BX=XB,\text{ and }AM=MB,$$
for all $X\in\frak{sl}_2(\mathbb{C})$ and all $M\in\mathbb{C}^{2\times2}$. This yields that $A=B=\pm I_2$. Hence, $\ker\iota=\langle(-I_2,-I_2,0,0)\rangle$. Therefore, $\iota$ induces a group isomorphism
\begin{align*}
\mathrm{Aut}_{\widehat{\mathcal{D}}\text{-conf}}(\mathcal{A}_{\widehat{\mathcal{D}}})&\cong
\frac{\left(\mathbf{SL}_2(\widehat{D})\times\mathbf{SL}_2(\widehat{D})\times\mathbf{G}_a(\widehat{D})\right)
\rtimes\Bbb{Z}/2\Bbb{Z}}{\langle(-I_2,-I_2,0,0)\rangle}\\
&\cong\left(\frac{\mathbf{SL}_2(\widehat{D})\times\mathbf{SL}_2(\widehat{D})}{\langle(-I_2,-I_2)\rangle}
\times\mathbf{G}_a(\widehat{D})\right)\rtimes\Bbb Z/2\Bbb Z.
\end{align*}
\end{proof}

\begin{remark}
\label{rmk:autoC}
The above theorem gives a precise description of the automorphism group of the $\widehat{\mathcal{D}}$--conformal superalgebra $\mathcal{A}_{\widehat{\mathcal{D}}}$. Using the same reasoning, we also can obtain the automorphism group of the $\Bbb{C}$--conformal superalgebra $\mathcal{A}$. In fact,
$$\mathrm{Aut}_{\Bbb{C}\text{-conf}}(\mathcal{A})\cong
\left(\frac{\mathbf{SL}_2(\Bbb{C})\times\mathbf{SL}_2(\Bbb{C})}{\langle(-I_2,-I_2)\rangle}
\times\mathbf{G}_a(\Bbb{C})\right)\rtimes\Bbb Z/2\Bbb Z.$$
\end{remark}

\section{Classification of twisted loop conformal superalgebras}
\label{sec:twistedform}

The classification of twisted loop conformal superalgebras based on $\mathcal{A}$ will be completed in this section. We firstly compute the non-abelian cohomology set $\mathrm{H}^1_{\mathrm{ct}}\left(\widehat{\mathbb{Z}},
\mathrm{Aut}_{\widehat{\mathcal{D}}\text{-conf}}(\mathcal{A}_{\widehat{\mathcal{D}}})\right)$, which yields the classification of twisted loop conformal superalgebras based on $\mathcal{A}$ up to isomorphisms of differential conformal superalgebras over $\mathcal{D}$. Then we deduce the classification up to isomorphisms of conformal superalgebras over $\Bbb{C}$ through  centroid considerations as in \cite{CP2011} and \cite{KLP}.

\begin{proposition}
\label{prop:Diso}
Every $\widehat{\mathcal{D}}/\mathcal{D}$--form of $\mathcal{A}_{\mathcal{D}}$ is isomorphic to either $\mathcal{L}(\mathcal{A},\mathrm{id})$ or $\mathcal{L}(\mathcal{A},\omega)$ as differential conformal superalgebras over $\mathcal{D}$.
\end{proposition}
\begin{proof}
Based on \cite[Theorem~2.16 and Proposition~2.29]{KLP}, $\widehat{\mathcal{D}}/\mathcal{D}$--forms of $\mathcal{A}_{\mathcal{D}}$ are parameterized by the continuous non-abelian cohomology set $\mathrm{H}^1_{\mathrm{ct}}\left(\widehat{\mathbb{Z}},\mathrm{Aut}_{\widehat{\mathcal{D}}\text{-conf}}(\mathcal{A}_{\widehat{\mathcal{D}}})\right)$, where $\widehat{\mathbb{Z}}:=\lim\limits_{\longleftarrow}\Bbb{Z}/m\Bbb{Z}$ and the continuous action of $\widehat{\Bbb{Z}}$ on $\mathrm{Aut}_{\widehat{\mathcal{D}}\text{-conf}}(\mathcal{A}_{\widehat{\mathcal{D}}})$ is induced by the continuous action of $\widehat{\Bbb{Z}}$ on $\widehat{D}$ given by ${}^1t^{p/q}=\zeta_q^pt^{p/q}$. Hence, the crucial point of the proof is to compute the cohomology set $\mathrm{H}^1_{\mathrm{ct}}\left(\widehat{\mathbb{Z}},\mathrm{Aut}_{\widehat{\mathcal{D}}\text{-conf}}(\mathcal{A}_{\widehat{\mathcal{D}}})\right)$.

By Theorem~\ref{autoTHM}, there is a split short exact sequence of groups
\begin{equation}
1\rightarrow\mathrm{G}\rightarrow\mathrm{Aut}_{\widehat{\mathcal{D}}\text{-conf}}(\mathcal{A}_{\widehat{\mathcal{D}}})
\rightarrow\Bbb{Z}/2\Bbb{Z}\rightarrow1,\label{exact}
\end{equation}
where $$\mathrm{G}:=\mathrm{G}_1
\times\mathbf{G}_a(\widehat{D}),\text{ and }\mathrm{G}_1:=\frac{\mathbf{SL}_2(\widehat{D})\times\mathbf{SL}_2(\widehat{D})}{\langle(-I_2,-I_2)\rangle}.$$
We observe that $\widehat{\Bbb{Z}}$ continuously acts on $\mathrm{G}$ through the action on $\widehat{D}$ and $\widehat{\Bbb{Z}}$ acts on $\Bbb{Z}/2\Bbb{Z}$ trivially. With these $\widehat{\Bbb{Z}}$--actions, the homomorphisms in (\ref{exact}) are all $\widehat{\Bbb{Z}}$--equivariant. Hence, the exact sequence (\ref{exact}) induces an exact sequence of continuous non-abelian cohomology sets
\begin{equation}
\mathrm{H}^1_{\mathrm{ct}}(\widehat{\Bbb{Z}},\mathrm{G})\longrightarrow
\mathrm{H}^1_{\mathrm{ct}}\left(\widehat{\Bbb{Z}},
\mathrm{Aut}_{\widehat{\mathcal{D}}\text{-conf}}(\mathcal{A}_{\widehat{\mathcal{D}}})\right)
\overset{\rho}{\longrightarrow}\mathrm{H}^1_{\mathrm{ct}}(\widehat{\Bbb{Z}},\Bbb{Z}/2\Bbb{Z}).
\label{eq:exact2}
\end{equation}

Since the exact sequence (\ref{exact}) is split, $\rho$ has a section, and so $\rho$ is surjective. Recall that $\widehat{\Bbb{Z}}$ acts on $\Bbb{Z}/2\Bbb{Z}$ trivially, we have $\mathrm{H}^1_{\mathrm{ct}}(\widehat{\Bbb{Z}},\Bbb{Z}/2\Bbb{Z})\cong\Bbb{Z}/2\Bbb{Z}=\{[0],[1]\}$.
Since (\ref{eq:exact2}) is exact, the fiber of $\rho$ over $[0]$ is measured by $\mathrm{H}^1_{\mathrm{ct}}(\widehat{\Bbb{Z}},\mathrm{G})$. To compute $\mathrm{H}^1_{\mathrm{ct}}(\widehat{\Bbb{Z}},\mathrm{G})$, we observe that $\widehat{\Bbb{Z}}$ piecewise acts on $\mathrm{G}=\mathrm{G}_1\times\mathbf{G}_a(\widehat{D})$. It follows
\begin{equation}\mathrm{H}^1_{\mathrm{ct}}(\widehat{\Bbb{Z}},\mathrm{G})=\mathrm{H}^1_{\mathrm{ct}}(\widehat{\Bbb{Z}},\mathrm{G}_1)
\times \mathrm{H}^1_{\mathrm{ct}}(\widehat{\Bbb{Z}},\mathbf{G}_a(\widehat{D})).\label{eq:ntch1}
\end{equation}
The group $\mathrm{G}_1$ fits into an exact sequence of groups
$$1\rightarrow\Bbb{Z}/2\Bbb{Z}\rightarrow \mathbf{SL}_2(\widehat{D})\times\mathbf{SL}_2(\widehat{D})\rightarrow\mathrm{G}_1\rightarrow1.$$
It yields an exact sequence of pointed sets
$$\mathrm{H}^1_{\mathrm{ct}}(\widehat{\Bbb{Z}},\Bbb{Z}/2\Bbb{Z})
\rightarrow\mathrm{H}^1_{\mathrm{ct}}(\widehat{\Bbb{Z}},\mathbf{SL}_2(\widehat{D})\times\mathbf{SL}_2(\widehat{D}))
\rightarrow\mathrm{H}^1_{\mathrm{ct}}(\widehat{\Bbb{Z}},\mathrm{G}_1)
\rightarrow\mathrm{H}^2_{\mathrm{ct}}(\widehat{\Bbb{Z}},\Bbb{Z}/2\Bbb{Z}).$$

Since $\mathbf{SL}_2\times\mathbf{SL}_2$ is a semi-simple group scheme, $\mathrm{H}^1_{\mathrm{ct}}\left(\widehat{\Bbb{Z}},\mathbf{SL}_2(\widehat{D})\times\mathbf{SL}_2(\widehat{D})\right)$ can be identified with the non-abelian \'{e}tale cohomology $\mathrm{H}^1_{\text{\'{e}t}}(D,\mathbf{SL}_2\times\mathbf{SL}_2)$ by \cite[Corollary~2.16]{GP}, which vanishes according to \cite[Theorem~3.1]{P}. Hence, $\mathrm{H}^1_{\mathrm{ct}}\left(\widehat{\Bbb{Z}},\mathbf{SL}_2(\widehat{D})\times\mathbf{SL}_2(\widehat{D})\right)=0$. $\mathrm{H}^2_{\mathrm{ct}}(\widehat{\Bbb{Z}},\Bbb{Z}/2\Bbb{Z})$ also vanishes since it can be identified with $\mathrm{H}^2_{\text{\'{e}t}}(D,\boldsymbol{\mu}_2)$, which is  the 2-torsion of the Brauer group $\mathrm{H}^2_{\text{\'{e}t}}(D,\mathbf{G}_m) = 1$. Therefore, \begin{equation}\mathrm{H}^1(\widehat{\Bbb{Z}},\mathrm{G}_1)=0.\label{eq:ntch2}\end{equation}

From \cite[I.2.2, Proposition~8]{S}, we deduce that $\mathrm{H}^1_{\mathrm{ct}}(\widehat{\Bbb{Z}},\mathbf{G}_a(\widehat{D}))
=\lim\limits_{\longrightarrow}\mathrm{H}^1_{\mathrm{ct}}(\Bbb{Z}/m\Bbb{Z},\mathbf{G}_a(D_m))$. Since $D_m/D$ is a Galois extension with Galois group $\Bbb{Z}/m\Bbb{Z}$, $\mathrm{H}^1_{\mathrm{ct}}(\Bbb{Z}/m\Bbb{Z},\mathbf{G}_a(D_m))=\mathrm{H}^1_{\text{\'{e}t}}(D_m/D,\mathbf{G}_{a,D})$ (see \cite[Remark~2.27]{KLP} for details). Now, $\mathrm{H}^1_{\text{\'{e}t}}(D_m/D,\mathbf{G}_{a,D})$ can be viewed as a subset of $\mathrm{H}^1_{\text{\'{e}t}}(D,\mathbf{G}_{a})$, which vanishes because because our base scheme, namely $\rm{Spec}(\it D)$, is affine (see \cite{DG} III.4.6.6). Hence,
\begin{equation}\mathrm{H}^1_{\mathrm{ct}}(\widehat{\Bbb{Z}},\mathbf{G}_a(\widehat{D}))=0.\label{eq:ntch3}\end{equation}
Summarizing (\ref{eq:ntch1}), (\ref{eq:ntch2}), and (\ref{eq:ntch3}), we obtain
$\mathrm{H}^1_{\mathrm{ct}}(\widehat{\Bbb{Z}},\mathrm{G})=0$, i.e., the fiber of $\rho$ over $[0]$ contains exactly one element.

Next we consider the fiber of $\rho$ over $[1]$. Twisting the $\widehat{\Bbb{Z}}$--groups in (\ref{exact}) with respect to the cocycle $\frak{z}:\widehat{\Bbb{Z}}\mapsto\mathrm{Aut}_{\widehat{\mathcal{D}}\text{-conf}}(\mathcal{A}_{\widehat{\mathcal{D}}}), 1\mapsto \omega$, we deduce that the fiber of $\rho$ over $[1]$ is measured by $\mathrm{H}^1_{\mathrm{ct}}(\widehat{\Bbb{Z}},{}_{\frak{z}}\mathrm{G})$. As $\widehat{\Bbb{Z}}$--groups, ${}_{\frak{z}}\mathrm{G}={}_{\frak{z}}\mathrm{G}_1\times{}_{\frak{z}}\mathbf{G}_a(\widehat{D})$. Hence, we also have
\begin{equation}
\mathrm{H}^1_{\mathrm{ct}}(\widehat{\Bbb{Z}},{}_{\frak{z}}\mathrm{G})= \mathrm{H}^1_{\mathrm{ct}}(\widehat{\Bbb{Z}},{}_{\frak{z}}\mathrm{G}_1)\times \mathrm{H}^1_{\mathrm{ct}}(\widehat{\Bbb{Z}},{}_{\frak{z}}\mathbf{G}_a(\widehat{D})).\label{eq:tch1}
\end{equation}

To compute $\mathrm{H}^1_{\mathrm{ct}}(\widehat{\Bbb{Z}},{}_{\frak{z}}\mathrm{G}_1)$, we also have an exact sequence
$$1\rightarrow{}_{\frak{z}}(\Bbb{Z}/2\Bbb{Z})\rightarrow {}_{\frak{z}}(\mathbf{SL}_2(\widehat{D})\times\mathbf{SL}_2(\widehat{D}))\rightarrow{}_{\frak{z}}\mathrm{G}_1\rightarrow1.$$
Since $\omega$ trivially acts on the subgroup $\langle(-I_2,-I_2)\rangle$ of $\mathbf{SL}_2(\widehat{D})\times\mathbf{SL}_2(\widehat{D})$, it follows ${}_{\frak{z}}(\Bbb{Z}/2\Bbb{Z})=\Bbb{Z}/2\Bbb{Z}$. Hence, there is a long exact sequence
$$\mathrm{H}^1_{\mathrm{ct}}(\widehat{\Bbb{Z}},\Bbb{Z}/2\Bbb{Z})
\rightarrow\mathrm{H}^1_{\mathrm{ct}}(\widehat{\Bbb{Z}},{}_{\frak{z}}(\mathbf{SL}_2(\widehat{D})\times\mathbf{SL}_2(\widehat{D})))
\rightarrow\mathrm{H}^1_{\mathrm{ct}}(\widehat{\Bbb{Z}},{}_{\frak{z}}\mathrm{G}_1)
\rightarrow\mathrm{H}^2_{\mathrm{ct}}(\widehat{\Bbb{Z}},\Bbb{Z}/2\Bbb{Z}).$$
We have seen that $\mathrm{H}^2_{\mathrm{ct}}(\widehat{\Bbb{Z}},\Bbb{Z}/2\Bbb{Z})=0$. Further, by the same reasons given above $\mathrm{H}^1_{\mathrm{ct}}(\widehat{\Bbb{Z}},{}_{\frak{z}}(\mathbf{SL}_2(\widehat{D})\times\mathbf{SL}_2(\widehat{D})))$ can be identified with the non-abelian \'{e}tale cohomology $\mathrm{H}^1_{\text{\'{e}t}}(D,{}_{\frak{z}}(\mathbf{SL}_2\times\mathbf{SL}_2))$, which vanishes since ${}_{\frak{z}}(\mathbf{SL}_2\times\mathbf{SL}_2)$ is also a reductive group scheme over $D$.
Hence,
\begin{equation}
\mathrm{H}^1_{\mathrm{ct}}(\widehat{\Bbb{Z}},{}_{\frak{z}}\mathrm{G}_1)=0.\label{eq:tch2}
\end{equation}

To understand $\mathrm{H}^1_{\mathrm{ct}}(\widehat{\Bbb{Z}},{}_{\frak{z}}\mathbf{G}_a(\widehat{D}))$, we first observe that
${}_{\frak{z}}\mathbf{G}_a$ is a twisted form of $\mathbf{G}_a$ (more precisely of the $D$-group $\mathbf{G}_{a, D}$)  associated to the cocycle $\frak{z'}:\widehat{\Bbb{Z}}\rightarrow\mathrm{Aut}(\mathbf{G}_a(\widehat{D})),\bar{1}\mapsto-\mathrm{id},$ viewed in a natural way as an element of $\mathrm{H}^1_{\text{\'{e}t}}\big({\it D}, \mathbf{Aut}(\mathbf{G}_a)\big)$.  The natural ${\it D}$-group homomorphism $  \mathbf{G}_{\mathrm{m}} \rightarrow \mathbf{Aut}(\mathbf{G}_a)$ yields a map
$$\phi: \mathrm{H}^1_{\text{\'{e}t}}\big({\it D}, \mathbf{G}_m\big)
\rightarrow\mathrm\mathrm{H}^1_{\text{\'{e}t}}\big({\it D}, \mathbf{Aut}(\mathbf{G}_a)\big).$$
Since the class $[\frak{z}']$ of $\frak{z}'$ is visibly in the image of $\phi$ and $\mathrm{H}^1_{\text{\'{e}t}}(D,\mathbf{G}_{\mathrm{m}})= {\rm Pic}({\it D}) = 0,$
we deduce that ${}_{\frak{z}}\mathbf{G}_a$ is isomorphic to $\mathbf{G}_a$ (or rather $\mathbf{G}_{a, {\rm D}}$ to be precise). This  yields that
\begin{equation}
\mathrm{H}^1_{\mathrm{ct}}(\widehat{\Bbb{Z}},{}_{\frak{z}}\mathbf{G}_a(\widehat{D}))
\subset \mathrm{H}^1_{\text{\'{e}t}}(D, {_{\frak{z}}\mathbf{G}_a}) = \mathrm{H}^1_{\text{\'{e}t}}(D, \mathbf{G}_a)=0.\label{eq:tch3}
\end{equation}

From (\ref{eq:tch1}), (\ref{eq:tch2}), and (\ref{eq:tch3}), we deduce that $\mathrm{H}^1_{\mathrm{ct}}(\widehat{\Bbb{Z}},{}_{\frak{z}}\mathrm{G})=0$, i.e., the fiber of $\rho$ over $[1]$ also contains exactly one element.

Consequently, $\mathrm{H}^1_{\mathrm{ct}}\left(\widehat{\Bbb{Z}},
\mathrm{Aut}_{\widehat{\mathcal{D}}\text{-conf}}(\mathcal{A}_{\widehat{\mathcal{D}}})\right)$ contains exactly two elements, which correspond to
$\mathcal{L}(\mathcal{A},\mathrm{id})$ and $\mathcal{L}(\mathcal{A},\omega)$.
\end{proof}

Next, we proceed to compute the centroid of an arbitrary twisted loop conformal superalgebra based on $\mathcal{A}$. For an arbitrary conformal superalgebra $\mathcal{B}$ over $\Bbb{C}$, the centroid of $\mathcal{B}$ is
$$\mathrm{Ctd}_{\Bbb{C}}(\mathcal{B})=\{\chi\in\mathrm{End}_{\Bbb{C}\text{-smod}}(\mathcal{B})|\chi(a_{(n)}b)=a_{(n)}\chi(b),\forall a,b\in\mathcal{B},n\in\Bbb{Z}_+\},$$
where $\mathrm{End}_{\Bbb{C}\text{-smod}}(\mathcal{B})$ is the set of endomorphism of the super $\Bbb{C}$-module $\mathcal{B}$. If in addition $\mathcal{B}$ is also a differential conformal superalgebra over $\mathcal{D}$, there is a canonical map $D\rightarrow\mathrm{Ctd}_{\Bbb{C}}(\mathcal{B}), r\mapsto r_{\mathcal{B}}$, where $r_{\mathcal{B}}$ is the map $\mathcal{B}\rightarrow\mathcal{B},a\mapsto ra$ (see \cite{KLP} for details).

\begin{lemma}
\label{lem:weightvectors}
$\mathcal{A}_{\mathcal{D}_m}=
\mathrm{span}_{\mathbb{C}}\{v_{(0)}(\partial^{(\ell)}\LL\otimes1),v_{(1)}(\partial^{(\ell)}\LL\otimes1)|
v\in\mathcal{A}_{\mathcal{D}_m}\}$.
\end{lemma}
\begin{proof}
We denote the $\C$--vector space on the right hand side by $\mathcal{V}$. Let $v\in\mathcal{A}$ be a primary eigenvector having eigenvalue $\Delta$ with respect to $\LL$, i.e., $v$ satisfies
$$v_{(0)}\LL=(\Delta-1)\partial v,\quad v_{(1)}\LL=\Delta v,\quad v_{(k)}\LL=0, \text{ for }k\geqslant2.$$
Then we deduce that
$$v_{(k)}\partial^{(\ell-1)}\LL=\begin{cases}(\ell(\Delta-1)+k)\partial^{(\ell-k)}v,&\text{if }\ell\geqslant k,\\
0,&\text{if }\ell<k,
\end{cases}$$
for $k\geqslant0$ and $\ell\geqslant1$.

In the conformal superalgebra $\mathcal{A}$, $\Delta=2,1,1,\frac{3}{2},\frac{1}{2}$ if $a=\LL,\UU,\TT^{\pm}(X),\GG(M),\QQ(M)$ respectively, where $X\in\frak{sl}_2(\mathbb{C})$ and $M\in\mathbb{C}^{2\times2}$. We consider two cases.

Case I: $\Delta=\frac{1}{2}$. Then
\begin{align*}
(v\otimes f)_{(1)}&(\LL\otimes1)=\frac{1}{2}v\otimes f,\\
(v\otimes f)_{(0)}&(\partial^{(\ell)}\LL\otimes1)\\
&=\sum\limits_{k\geqslant0}(v_{(k)}\partial^{(\ell)}\LL)\otimes\delta_t^{(k)}(f)\\
&=\sum\limits_{k=0}^{\ell+1}(-\frac{1}{2}(\ell+1)+k)\partial^{(\ell+1-k)}v\otimes\delta_t^{(k)}(f)\\
&=-\frac{1}{2}(\ell+1)\partial^{(\ell+1)}v\otimes f+\sum\limits_{k=1}^{\ell+1}(-\frac{1}{2}(\ell+1)+k)\partial^{(\ell+1-k)}v\otimes\delta_t^{(k)}(f)
\end{align*}
for all $\ell\geqslant0$. By induction, we obtain $\partial^{(\ell)}v\otimes f\in\mathcal{V}$ for all $\ell\geqslant0$.

Case II: $\Delta\geqslant1$. Then
\begin{align*}
(v\otimes f)_{(1)}&(\LL\otimes1)=\Delta v\otimes f,\\
(v\otimes f)_{(1)}&(\partial^{(\ell)}\LL\otimes1)\\
&=\sum\limits_{k\geqslant0}(v_{(k+1)}\partial^{(\ell)}\LL)\otimes\delta_t^{(k)}(f)\\
&=\sum\limits_{k=0}^{\ell}((\ell+1)(\Delta-1)+k+1)\partial^{(\ell-k)}v\otimes\delta_t^{(k)}(f)\\
&=((\ell+1)(\Delta-1)+1)\partial^{(\ell)}v\otimes f+\sum\limits_{k=0}^{\ell}((\ell+1)(\Delta-1)+k+1)\partial^{(\ell-k)}v\otimes\delta_t^{(k)}(f).
\end{align*}
Since $\Delta\geqslant1$, $(\ell+1)(\Delta-1)+1>0$. By induction, $\partial^{(\ell)}v\otimes f\in\mathcal{V}$.

Recall that $\mathcal{A}_{\mathcal{D}_m}=\Bbb{C}[\partial]\otimes_{\Bbb{C}}V\otimes_{\Bbb{C}}D_m$, where $V=\mathrm{span}_{\mathbb{C}}\{\LL,\UU,\TT^{\pm}(X),\GG(M),\QQ(M)|X\in\frak{sl}_2(\Bbb{C}),M\in\Bbb{C}^{2\times2}\}$. We observe that $\mathcal{A}_{\mathcal{D}_m}$ is spanned by $\partial^{(\ell)}v\otimes f$ for $\ell\geqslant0, v\in V$, and $f\in D_m$. Hence, $\mathcal{A}_{\mathcal{D}_m}=\mathcal{V}$.
\end{proof}

\begin{proposition}
\label{prop:centroid}
Let $\mathcal{B}:=\mathcal{L}(\mathcal{A},\sigma)$ be the twisted loop conformal superalgebra based on $\mathcal{A}$ with respect to an automorphism $\sigma$ of order $m$. Then $\mathrm{Ctd}_{\mathbb{C}}(\mathcal{B})\cong D$.
\end{proposition}
\begin{proof}
For $f\in D$, there is an element $f_{\mathcal{B}}\in\mathrm{Ctd}_{\mathbb{C}}(\mathcal{B})$ given by $v\mapsto fv$. Hence, $D\subseteq\mathrm{Ctd}_{\mathbb{C}}(\mathcal{B})$. It suffices to show any element $\chi\in\mathcal{B}$ is of the form $f_{\mathcal{B}}$ for some $f\in D$.

We consider the $\Bbb{C}$--linear map
$$\pi:\mathcal{A}_{\mathcal{D}_m}\rightarrow\mathcal{A}_{\mathcal{D}_m},\quad
v\mapsto \frac{1}{m}\sum\limits_{i=0}^{m-1}(\sigma\otimes\psi)^i(v),$$
where $\psi:D_m\rightarrow D_m, t^{\frac{1}{m}}\mapsto\zeta_m^{-1}t^{\frac{1}{m}}$. Then $\mathcal{B}=\pi(\mathcal{A}_{\mathcal{D}_m})$.

It has been pointed out in Remark~\ref{rmk:autoC} that $\sigma=\theta_{A,B}\circ\tau_{\alpha}\circ\omega^{\varepsilon}$, where $A,B\in\mathbf{SL}_2(\mathbb{C})$, $\alpha\in\mathbb{C}$ and $\varepsilon\in\{0,1\}$. Hence,
$$\sigma(\LL)=\LL+\alpha\UU,\quad \sigma(\UU)=\pm\UU.$$
We consider two cases.

Case I: $\sigma(\UU)=\UU$. In this situation, $\sigma(\LL)=\LL$ since $\sigma$ is of finite order. Hence,
\begin{align*}
\pi(v_{(n)}(\partial^{(\ell)}\LL\otimes1))
&=\frac{1}{m}\sum\limits_{i=0}^{m-1}(\sigma\otimes\psi)^i(v_{(n)}(\partial^{(\ell)}\LL\otimes1))\\
&=\frac{1}{m}\sum\limits_{i=0}^{m-1}((\sigma\otimes\psi)^i(v))_{(n)}(\sigma\otimes\psi)^i(\partial^{(\ell)}\LL\otimes1)\\
&=\frac{1}{m}\sum\limits_{i=0}^{m-1}((\sigma\otimes\psi)^i(v))_{(n)}(\partial^{(\ell)}\LL\otimes1)\\
&=\pi(v)_{(n)}(\partial^{(\ell)}\LL\otimes1).
\end{align*}
By Lemma~\ref{lem:weightvectors}, we deduce
\begin{align*}
\mathcal{B}=\pi(\mathcal{A}_{\mathcal{D}_m})
&=\mathrm{span}_{\mathbb{C}}\{\pi(v)_{(0)}(\partial^{(\ell)}\LL\otimes1),
\pi(v)_{(1)}(\partial^{(\ell)}\LL\otimes1)|v\in\mathcal{A}_{\mathcal{D}_m},\ell\geqslant0\},\\
&=\mathrm{span}_{\mathbb{C}}\{u_{(0)}(\partial^{(\ell)}\LL\otimes1),
u_{(1)}(\partial^{(\ell)}\LL\otimes1)|u\in\mathcal{B},\ell\geqslant0\}.
\end{align*}

Let $\chi\in\mathrm{Ctd}_{\mathbb{C}}(\mathcal{B})$. We claim that there is a $f\in D$ such that $\chi(\LL\otimes1)=\LL\otimes f$. Consider
$$(\LL\otimes1)_{(1)}\chi(\LL\otimes1)=\chi((\LL\otimes1)_{(1)}(\LL\otimes1))=2\chi(\LL\otimes1).$$
Hence, $\chi(\LL\otimes1)\in\mathcal{B}\subseteq\mathcal{A}_m$ is an eigenvector of $(\LL\otimes1)_{(1)}$ with eigenvalue $2$. We can check that the eigenspace of $(\LL\otimes1)_{(1)}$ with eigenvalue $2$ in $\mathcal{A}_{\mathcal{D}_m}$ is spanned by $\LL\otimes f, \partial v_i\otimes f_i$, where $\{v_i\}_{i=1}^7$ is a $\Bbb{C}$--basis of $\mathrm{span}_{\mathbb{C}}\{\TT^{\pm}(X),\UU|X\in\frak{sl}_2(\mathbb{C})\}$ and $f,f_i\in D_m$. We write
$$\chi(\LL\otimes1)=\LL\otimes f+\sum\limits_{i=1}^7\partial v_i\otimes f_i.$$
Then
\begin{align*}
0&=\chi((\LL\otimes1)_{(2)}(\LL\otimes1))\\
&=(\LL\otimes1)_{(2)}\chi(\LL\otimes1)\\
&=(\LL\otimes1)_{(2)}\left(\LL\otimes f+\sum\limits_{i=1}^7\partial v_i\otimes f_i\right)\\
&=2\sum\limits_{i=1}^7v_i\otimes f_i.
\end{align*}
Since $\{v_i\}$ is linear independent over $\Bbb{C}$, $f_i=0$, and so $\chi(\LL\otimes1)=\LL\otimes f$. In particular, $f\in D$ since $\LL\otimes f\in\mathcal{B}$. Furthermore,
$$\ell\chi(\partial^{(\ell)}\LL\otimes1)=\chi((\LL\otimes1)_{(0)}(\partial^{(\ell-1)}\LL\otimes1))
=(\LL\otimes1)_{(0)}\chi(\partial^{(\ell-1)}\LL\otimes1),$$
for all $\ell\geqslant1$. By induction, $\chi(\partial^{(\ell)}\LL\otimes1)=\partial^{(\ell)}\LL\otimes f$. Hence, for $v\in\mathcal{B}$,
\begin{align*}
\chi(v_{(0)}(\partial^{(\ell)}\LL\otimes1))&=v_{(0)}\chi(\partial^{(\ell)}\LL\otimes1)=v_{(0)}(\partial^{(\ell)}\LL\otimes f)=f(v_{(0)}(\partial^{(\ell)}\LL\otimes1)).\\
\chi(v_{(1)}(\partial^{(\ell)}\LL\otimes1))&=v_{(1)}\chi(\partial^{(\ell)}\LL\otimes1)=v_{(1)}(\partial^{(\ell)}\LL\otimes f)=f(v_{(1)}(\partial^{(\ell)}\LL\otimes1)).
\end{align*}
It follows $\chi=f_{\mathcal{B}}$.

Case II: $\sigma(\UU)=-\UU$ and $\sigma(\LL)=\LL+\alpha\UU$. In this situation, we replace $\LL$ by $\LL':=\LL+\frac{\alpha}{2}\UU$ and replace $\GG^p$ by $\GG'^p:=\GG^p+\frac{\alpha}{2}\QQ^p$ for $p=1,2,3,4$. Then $\mathcal{A}$ is also a $\Bbb{C}[\partial]$--module generated by $\LL', \TT^{\pm i}, \UU, \GG'^p, \QQ^p, i=1,2,3, p=1,2,3,4$, and in $\mathcal{A}$, we have
\begin{align*}
&[\LL'{}_\lambda\LL']=(\partial+2\lambda)\LL',
&&[\LL'{}_\lambda\TT^{\pm i}]=(\partial+\lambda)\TT^{\pm i},
&&[\LL'{}_\lambda\UU]=(\partial+\lambda)\UU,\\
&[\LL'{}_\lambda\GG'^p]=(\partial+\frac{3}{2}\lambda)\GG'^p,
&&[\LL'{}_\lambda\QQ^p]=(\partial+\frac{1}{2}\lambda)\QQ^p.
\end{align*}
Similar arguments as in Lemma~\ref{lem:weightvectors} show that
$$\mathcal{A}_{\mathcal{D}_m}=\mathrm{span}_{\Bbb{C}}\{v_{(0)}(\partial^\ell\LL'\otimes1), v_{(1)}(\partial^\ell\LL'\otimes1)|v\in\mathcal{A}_{\mathcal{D}_m},\ell\geqslant0\}.$$
Following the same arguments as in Case I, we conclude that $\chi=f_{\mathcal{B}}$ for some $f\in D$.
\end{proof}

\begin{theorem}
Every twisted loop conformal superalgebra based on the large $N=4$ conformal superalgebra $\mathcal{A}$ is isomorphic to either $\mathcal{L}(\mathcal{A},\mathrm{id})$ or $\mathcal{L}(\mathcal{A},\omega)$ as conformal superalgebras over $\Bbb{C}$.
\end{theorem}
\begin{proof}
As pointed out in \cite{KLP} that every twisted loop conformal superalgebra $\mathcal{L}(\mathcal{A},\sigma)$ based on $\mathcal{A}$ is not only a conformal superalgebra over $\Bbb{C}$ but also a differential conformal superalgebra over $\mathcal{D}$. Moreover, $\mathcal{L}(\mathcal{A},\sigma)$ is a $\widehat{\mathcal{D}}/\mathcal{D}$--form of $\mathcal{L}(\mathcal{A},\mathrm{id})=\mathcal{A}_{\mathcal{D}}$ (\cite[Proposition~2.4]{KLP}). By Proposition~\ref{prop:Diso}, there are only two $\widehat{\mathcal{D}}/\mathcal{D}$--forms of $\mathcal{A}_{\mathcal{D}}$ up to isomorphisms of differential conformal superalgebras over $\mathcal{D}$. They are $\mathcal{L}(\mathcal{A},\mathrm{id})$ and $\mathcal{L}(\mathcal{A},\omega)$.

From Proposition~\ref{prop:centroid} and \cite[Corollary~2.36]{KLP}, we deduce that two twisted loop conformal superalgebra $\mathcal{L}(\mathcal{A},\sigma_1)$ and $\mathcal{L}(\mathcal{A},\sigma_2)$ based on $\mathcal{A}$ are isomorphic as differential conformal superalgebras over $\mathcal{D}$ if and only if $\mathcal{L}(\mathcal{A},\sigma_1)$ is isomorphic to $\mathcal{L}(\mathcal{A},\sigma_2)$ as conformal superalgebras over $\Bbb{C}$. Hence, every twisted loop conformal superalgebra based on $\mathcal{A}$ is isomorphic to either $\mathcal{L}(\mathcal{A},\mathrm{id})$ or $\mathcal{L}(\mathcal{A},\omega)$ as conformal superalgebras over $\Bbb{C}$.
\end{proof}

\section{The corresponding Lie superalgebra}
\label{sec:superalgebra}

As we have seen in Section~\ref{sec:intro}, every twisted loop conformal superalgebra $\mathcal{L}(\mathcal{A},\sigma)$ based on the large $N=4$ conformal superalgebra $\mathcal{A}$ yields an infinite dimensional Lie superalgebra
$\mathrm{Alg}(\mathcal{A},\sigma)$. In particular, the untwisted loop conformal superalgebra $\mathcal{L}(\mathcal{A},\mathrm{id})$ yields the Lie superalgbra
$\frak{g}$ described in Section~\ref{sec:conformalalgebra}, which is the centerless core of the large $N=4$ superconformal algebras given in \cite{STV}.

There is another twisted loop conformal superalgebra $\mathcal{L}(\mathcal{A},\omega)$ not isomorphic to $\mathcal{L}(\mathcal{A},\mathrm{id})$.
$\mathcal{L}(\mathcal{A},\omega)$ gives rise to another Lie superalgebra $\mathrm{Alg}(\mathcal{A},\omega)$. In the rest part of this section, we will explicitly state the generators and relations of $\mathrm{Alg}(\mathcal{A},\omega)$.

Recall that $\mathcal{L}(\mathcal{A},\omega)=\oplus_{\ell\in\Bbb{Z}}\mathcal{A}_{\ell}\otimes \Bbb{C}t^{\ell/2}$, where $\mathcal{A}_{\ell}=\{a\in\mathcal{A}|\omega(a)=(-1)^{\ell}a\}$. It can be directly computed that
\begin{align*}
\mathcal{A}_{\ell}=\begin{cases}
\mathbb{C}[\partial]\otimes_{\Bbb{C}}\mathrm{span}_{\mathbb{C}}\{\LL,\TT^{+i}+\TT^{-i},\GG^i,\QQ^4|i=1,2,3\},&\text{if }\ell\text{ is even,}\\
\mathbb{C}[\partial]\otimes_{\Bbb{C}}
\mathrm{span}_{\mathbb{C}}\{\UU,\TT^{+i}-\TT^{-i},\GG^4,\QQ^i|i=1,2,3\},&\text{if }\ell\text{ is odd.}
\end{cases}
\end{align*}
As a vector space, $\mathrm{Alg}(\mathcal{A},\omega)=\mathcal{L}(\mathcal{A},\omega)/\widehat{\partial}\mathcal{L}(\mathcal{A},\omega)$.
We use $\bar{v}$ to denote the image of an element $v\in\mathcal{L}(\mathcal{A},\omega)$ in $\mathrm{Alg}(\mathcal{A},\omega)$. Then the following elements,
\begin{align*}
&\LL_m=\overline{\LL\otimes t^{m+1}},
&&\TT^i_m=\overline{(\TT^{+i}+\TT^{-i})\otimes t^m},
&&\JJ^i_r=\overline{(\TT^{+i}-\TT^{-i})\otimes t^r},
&&\UU_r=\overline{\UU\otimes t^r},\\
&\GG^i_r=\overline{\GG^i\otimes t^{r+1/2}},
&&\Phi_m=\overline{\GG^4\otimes t^{m+1/2}},
&&\QQ^i_m=\overline{\QQ^i\otimes t^{m-1/2}},
&&\Psi_r=\overline{\QQ^4\otimes t^{r-1/2}}.
\end{align*}
for $i=1,2,3$, $m\in\Bbb{Z}$ and $r\in1/2+\Bbb{Z}$, form a basis of $\mathrm{Alg}(\mathcal{A},\sigma)$. The super-bracket on $\mathrm{Alg}(\mathcal{A},\omega)$ can be written as:
\begin{center}
\begin{tabular}{ll}
$[\LL_m,\LL_n]=(m-n)\LL_{m+n},$
&$[\LL_m,\UU_s]=-s\UU_{m+s},$\\
$[\LL_m,\TT^i_n]=-n\TT^i_n,$
&$[\LL_m,\JJ^i_s]=-s\JJ^i_{m+s},$\\
$[\LL_m,\GG^i_s]=(m/2-s)\GG^i_{m+s},$
&$[\LL_m,\Phi_n]=(m/2-n)\Phi_{m+n},$\\
$[\LL_m,\QQ^i_n]=-(m/2+n)\GG^i_{m+n},$
&$[\LL_m,\Psi_s]=-(m/2+s)\Psi_{m+s},$\\
$[\UU_r,\UU_s]=0,$
&$[\UU_r,\TT^i_n]=[\UU_r,\JJ^i_s]=0,$\\
$[\UU_r,\GG^i_s]=r\QQ^i_{r+s},$
&$[\UU_r,\Phi_n]=r\Psi_{r+n},$\\
$[\UU_r,\QQ^i_n]=0,$
&$[\UU_r,\Psi_s]=0,$\\
$[\TT^i_m,\TT^j_n]=\epsilon_{ijk}\TT^k_{m+n},$
&$[\TT^i_m,\JJ^j_s]=\epsilon_{ijk}\JJ^k_{m+s},$\\
$[\TT^i_m,\GG^j_s]=\epsilon_{ijk}\GG^k_{m+s}-m\delta_{ij}\Psi_{m+s},$
&$[\TT^i_m,\Phi_n]=m\QQ^i_{m+n},$\\
$[\TT^i_m,\QQ^j_n]=\epsilon_{ijk}\QQ^k_{m+n},$
&$[\TT^i_m,\Psi_s]=0,$\\
$[\JJ^i_r,\JJ^j_s]=\epsilon_{ijk}\TT^k_{r+s},$
&\\
$[\JJ^i_r,\GG^j_s]=\delta_{ij}\Phi_{r+s}-r\epsilon_{ijk}\QQ^k_{r+s},$
&$[\JJ^i_r,\Phi_n]=-\GG^i_{r+n},$\\
$[\JJ^i_r,\QQ^j_n]=\delta_{ij}\Psi_{r+n},$
&$[\JJ^i_r,\Psi_s]=-\QQ^i_{r+s},$\\
$[\GG^i_r,\GG^j_s]=2\delta_{ij}\LL_{r+s}-\epsilon_{ijk}(r-s)\TT^k_{r+s},$
&$[\GG^i_r,\Phi_n]=(n-r)\JJ^i_{r+n},$\\
$[\QQ^i_m,\GG^j_s]=\delta_{ij}\UU_{m+s}+\epsilon_{ijk}\JJ^k_{m+s},$
&$[\QQ^i_m,\Phi_n]=\TT^i_{m+n},$\\
$[\QQ^i_m,\QQ^j_n]=[\QQ^i_m,\Psi_s]=[\Psi_r,\Psi_s]=0,$
&$[\Phi_m,\Phi_n]=2\LL_{m+n},$\\
$[\Psi_r,\GG^i_s]=-\TT^i_{r+s},$&$[\Psi_r,\Phi_n]=\UU_{r+n},$
\end{tabular}
\end{center}
for $i,j=1,2,3$, $m,n\in\Bbb{Z}$ and $r,s\in1/2+\Bbb{Z}$.

In fact, the twisted large $N=4$ superconformal algebra described in \cite{V} is isomorphic to a central extension of the Lie superalgebra $\mathrm{Alg}(\mathcal{A},\omega)$.

\begin{remark}\label{miss}
Let $\mathcal{A}(\gamma)$ be conformal superalgebra associated to $\frak{g}(\gamma)$.  From the relation
$$[\LL_\lambda\LL]=(\partial+2\lambda)\LL+\frac{1}{12}\lambda^3c,\text{ and }[\LL_\lambda\UU]=(\partial+\lambda)\UU-\frac{1}{3}\left(\gamma-\frac{1}{2}\right)\lambda^2c,$$
in $\mathcal{A}(\gamma)$, we observe that the automorphisms $\tau_f$ with $f\in\Bbb{C}$ and $\omega$ of $\mathcal{A}$ constructed in Section~\ref{sec:automorphism} can not be lifted to an automorphism of $\mathcal{A}(\gamma)$ if $\gamma\neq \frac{1}{2}$. This would seem to justify the absence of one-dimensional central extensions of $\mathrm{Alg}(\mathcal{A},\omega)$ in \cite{STV} when $\gamma\neq \frac{1}{2}$.

In contrast, both the automorphism $\tau_f$ and $\omega$ of $\mathcal{A}$ can be lifted to automorphisms $\hat{\tau}_f$ and $\hat{\omega}$ of $\mathcal{A}(\frac{1}{2})$. The action of $\hat{\tau}_f$ and $\hat{\omega}$ on $\mathcal{A}(\gamma)$ is explicitly given by:
\begin{align*}
&\hat{\tau}_f(\LL)=\LL+f\UU-\frac{f^2}{6}c,&&\hat{\tau}_f(\UU)=\UU-\frac{f}{3}c,&&\hat{\tau}_f(\TT^{\pm i})=\TT^{\pm i},\\
&\hat{\tau}_f(\GG(M))=\GG(M)+\QQ(fM),&&\hat{\tau}_f(\QQ(M))=\QQ(M),&&\hat{\tau}_f(c)=c\\
&\hat{\omega}(\LL)=\LL,&&\hat{\omega}(\UU)=-\UU,&&\hat{\omega}(\TT^{\pm i})=\TT^{\mp i},\\
&\hat{\omega}(\GG(M))=\GG(M^\dag),&&\hat{\omega}(\QQ)=-\QQ(M^\dag),&&\hat{\omega}(c)=c.
\end{align*}
for $i=1,2,3, M\in\Bbb{C}^{2\times2}$.

There is a natural (injective) group homomorphism from $\mathrm{Aut}_{\Bbb{C}\text{-conf}}\big(\mathcal{A}(\gamma)\big) \to \mathrm{Aut}\big(\frak{g}(\gamma)\big)$. The above considerations applied to the case $\gamma = \frac{1}{2}$ show that  the group of automorphisms of $\frak{g}(\frac{1}{2})$ is indeed larger than the one described in the Physics literature.
\end{remark}

\end{document}